\documentclass[12pt,leqno]{amsart}
\usepackage{amsmath,amssymb,amsfonts}
\usepackage[all]{xy}

\usepackage[T5,T1]{fontenc}
\usepackage{mathpazo}
\usepackage{hyperref}
\usepackage{a4wide}

\theoremstyle{plain}
\newtheorem{thm}{Theorem}[section]
\newtheorem{lem}[thm]{Lemma}

\newtheorem{prop}[thm]{Proposition}

\theoremstyle{definition}
\newtheorem{defn}[thm]{Definition}

\newtheorem{rmk}[thm]{Remark}

\newcommand{\Spec}{{\rm Spec \,}}

\newcommand{\Gal}{{\rm Gal}}

\newcommand{\A}{{\mathbb A}}

\newcommand{\G}{{\mathbb G}}
\renewcommand{\H}{{\mathbb H}}

\def\NDT{{\fontencoding{T5}\selectfont Nguy\~ \ecircumflex n Duy T\^an}}
\def\NQT{{\fontencoding{T5}\selectfont Nguy\~ \ecircumflex n Qu\'\ocircumflex{}c Th\'\abreve{}ng} } 
\begin{document}
\title{Special unipotent groups are split}
\begin{abstract} 
We show that over any field $k$, a smooth unipotent algebraic $k$-group is special if and only if it is $k$-split.
\end{abstract} 

 \author{ \NDT}
\address{Institute of Mathematics, Vietnam Academy of Science and Technology, 18 Hoang Quoc Viet, 10307, Hanoi - Vietnam } 
\email{duytan@math.ac.vn}
\thanks{NDT is partially supported  by Vietnam National Foundation for Science and Technology Development (NAFOSTED) under grant number 101.04-2017.09}
\maketitle
\section{Introduction}
An {\it algebraic group} over a field $k$ is a $k$-group scheme of finite type over $k$. 
The smooth affine algebraic $k$-groups considered here are the same as linear algebraic groups defined over $k$ in the sense of \cite{Bo}.  
Recall that an affine algebraic $k$-group $G$ is called {\it unipotent} if $G_{\bar{k}}$ (the base change of $G$ to a fixed algebraic closure $\bar{k}$ of $k$) admits a finite composition series over $\bar{k}$ with each successive quotient isomorphic to a $\bar{k}$-subgroup of the additive group $\G_a$.
It is well-known  that an affine algebraic $k$-group $G$ is unipotent if and only if is $k$-isomorphic to a closed $k$-subgroup scheme of the group $T_n$ consisting of upper triangular $n\times n$ matrices with all 1 on the main diagonal, for some $n$. 

A smooth unipotent algebraic group $G$ over a field $k$ is called $k$-{\it split} if either it is trivial or it  admits a composition series by $k$-subgroups with successive quotients are $k$-isomorphic to the additive group $\G_a$. 
We say that $G$ is $k$-{\it wound} if every morphism of $k$-schemes $\A^1_k\to G$ is constant.

Let $G$ be a smooth affine algebraic  group over a field $k$. We say that  $G$ is {\it special}, if
for any field extension $L/k$, every $G$-torsor over $\Spec L$ is trivial, i.e., if for any field extension $L/k$, the Galois cohomology set $H^1(L,G)$ is trivial. Special groups have been introduced by Serre in \cite{Se1}. Over algebraically closed fields, they have been classified by Grothendieck \cite{Gro}.  

Suppose that $G$ is a split smooth unipotent group over a field $k$. By induction on $\dim G$ and the fact that $H^1(L,\G_a)=0$ for every field extension $L$ over $k$, we see that $G$ is special. It is also well-known that over a perfect field $k$, every smooth connected unipotent group $G$ is $k$-split (see e.g. \cite[Chapter V, Corollary 15.5 (ii)]{Bo})). Hence over a perfect field, a  smooth unipotent group is special if and only if it is $k$-split. (Note that a special algebraic group  is always connected \cite[\S 4, Theorem 1]{Se1}.  For the reader's convenience, we will provide an alternative proof for this fact when the group is unipotent, see Lemma~\ref{lem:non-connected}.)
The following result is the main result of this note. 
\begin{thm}
\label{thm:main}
Let $G$ be a smooth unipotent algebraic group over a field $k$. Then $G$ is special if and only if $G$ is $k$-split.
\end{thm}
Our main result provides an affirmative answer to \cite[Question 1.4]{T} on a simple characterization of special unipotent groups over an arbitrary field.  M. Huruguen  \cite{Hu} provides a general classification result for special reductive groups over an arbitrary field. He also applies this result to obtain explicit classifications for several classes of special reductive groups such as special semisimple groups, special reductive groups of inner type and special quasisplit reductive groups.  R. Achet \cite{A}, as a corollary of his main result, proves Theorem~\ref{thm:main} in the particular case of unipotent groups of dimension 1, by a different method.
\\
\\
{\bf Acknowledgement:} We would like to thank  Michel Brion and \NQT  for interesting discussions and suggestions. Before finding the current proof for Lemma~\ref{lem:01}, Michel Brion showed us a proof for this lemma using a different technique. We are very grateful to him for sharing us his proof. We would also like to thank Rapha\"{e}l Achet for sending us his paper \cite{A}. We are also grateful to the referee for his/her comments and valuable suggestions which we used to improve our exposition. 
\section{Proof of the main result}
\subsection{Some results of Tits on the structure of unipotent algebraic groups}
We first recall some results of Tits concerning the structure  of  unipotent algebraic groups over an arbitrary  (especially imperfect) field of positive characteristic, see \cite[Chapter V]{Oe} and \cite[Appendix B]{CGP}.

Let $G$ be a smooth connected unipotent algebraic group over a field $k$ of characteristic $p>0$. 
There is a maximal $k$-split $k$-subgroup $G_s$, and it enjoys the following properties: it is normal in $G$, the quotient $G/G_s$ is $k$-wound and the formation of $G_s$ commutes with separable (not necessarily algebraic) extensions, see \cite[Chapter V, 7]{Oe} and \cite[Theorem B.3.4]{CGP}. 

Also there exists a maximal central smooth connected $p$-torsion $k$-subgroup of $G$. This group is called the {\it cckp-kernel} of $G$ and denoted by $cckp(G)$ or $\kappa(G)$. Here $\dim (\kappa(G))>0$ if $G$ is not finite. 

The following statements are equivalent:
\begin{enumerate}
\item $G$ is wound over $k$,
\item $\kappa(G)$ is wound over $k$.
\end{enumerate}
If the two equivalent statements are satisfied then $G/\kappa(G)$ is also wound over $k$ (\cite[Chapter V, 3.2]{Oe};  \cite[Appendix B, B.3]{CGP}).

\begin{defn}
 Let $k$ be a field of characteristic $p>0$. A polynomial $P\in k[T_1,\ldots,T_r]$ is a $p$-{\it polynomial} if every monomial appearing in $P$ has the form $c_{ij} T_i^{p^j}$ for some $c_{ij}\in k$; that is $P=\sum_{i=1}^r P_i(T_i)$ with $P_i(T_i)=\sum_{j} c_{ij} T_i^{p^j}\in k[T_i]$. 

A $p$-polynomial $P\in k[T_1,\ldots,T_r]$ is called {\it separable} if it contains at least a non-zero monomial of degree 1.

If $P=\sum_{i=1}^r P_i(T_i)$ is a $p$-polynomial over $k$ in $r$ variables, then the {\it principal part} of $P$ is the sum of the leading terms of the $P_i$.
\end{defn}
\begin{prop}[{see \cite[Ch. V, 6.3, Proposition]{Oe} and \cite[Proposition B.1.13]{CGP}}]
\label{prop:Tits1}
 Let $k$ be an infinite field of characteristic $p>0$. Let $G$ be a smooth  unipotent algebraic $k$-group of dimension $n$. Assume that $G$ is commutative and annihilated by $p$. Then $G$ is isomorphic (as a $k$-group) to the zero scheme of a separable nonzero $p$-polynomial $P$ over $k$. If we assume further that $G$ is $k$-wound then we can choose the polynomial $P$ such that its principal part vanishes nowhere over $k^{n+1}\setminus\{0\}$.
\end{prop}
\subsection{Wound unipotent groups}
In this subsection we always assume that $G$ is a smooth unipotent algebraic group over a field of characteristic $p>0$.
\begin{lem}
\label{lem:01} Suppose that $G$ is nontrivial, $k$-wound, commutative and annihilated by $p$. Then $G$ is not special.
\end{lem}
\begin{proof} 
Since $G$ is  nontrivial and $k$-wound, this implies  that  the ground field $k$ is imperfect. In particular $k$ is infinite.

By Proposition~\ref{prop:Tits1}, $G$ is $k$-isomorphic to a $k$-subgroup of $\G_a^r$, where $r=\dim G +1$, which is the zero set of a separable $p$-polynomial $P(T_1,\ldots,T_r)\in k[T_1,\ldots,T_r]$, whose principal part $P_{princ}=\sum_{i=1}^r c_i T_i^{p^{m_i}}$ vanishes nowhere over $k^r\setminus \{0\}$. From the exact sequence of $k$-groups
\[
0 \to G\to \G^r_a\stackrel{P}{\to} \G_a\to 0,
\]
one has $H^1(L,G)\simeq L/P(L\times\cdots \times L)$, for every field extension $L$ over $k$. 

Let $k((t))$ be the field of formal Laurent series in  variable $t$, and let $v$ be the $t$-adic valuation on 
$k((t))$. 

{\bf Claim}: $t^{-1}\not \in P(k((t))\times\cdots\times k((t)))$.

{\it Proof of Claim}: Suppose that $t^{-1}=P(\alpha_1,\ldots,\alpha_r)$ for some $\alpha_1,\ldots,\alpha_r\in k((t))$. Then there is some $i$ such that $v(\alpha_i)\leq -1$. We set 
\[
m=\min_{i=1,\ldots,r} v(\alpha_i^{p^{m_i}}).
\] 
Then $m\leq -p<-1$. We set 
\[
I=\{i\mid v(\alpha_i^{p^{m_i}})=m\}\subseteq \{1,\ldots,r\}.
\]
Then the coefficient of $t^m$ in $P(\alpha_1,\ldots,\alpha_r)$ is $\sum_{i\in I} c_i a_i^{p^{m_i}}$, where for each $i\in I$, 
\[\alpha_i= a_i t^{v(\alpha_i)}+ (\text{higher degree terms}),\] and $a_i\not =0$. Hence 
\[
\sum_{i\in I} c_i a_i^{p^{m_i}}=0,\]
which contradicts  the fact that $P_{princ}$ vanishes nowhere over $k^r\setminus \{0\}$.

The above claim implies that $H^1(k((t)),G)$ is nontrivial. Hence $G$ is not special.
\end{proof}

\begin{lem}
\label{lem:surj}
 Let $k$ be a field, $G$ a smooth affine algebraic $k$-group. Let $U$ be a normal unipotent $k$-subgroup of $G$. Then the natural map 
$$\varphi: H^1(k,G)\to H^1(k,G/U)$$
is surjective.

Furthermore, if in addition that $U$ is $k$-split then $\varphi$ is a functorial bijection.
\end{lem}

\begin{proof}
See \cite[Chapter IV, 2.2, Remark 3]{Oe} for the first statement.

See \cite[Lemma 7.3]{GM} for the second statement.
\end{proof}

\begin{lem}
\label{lem:02} Suppose that $G$ is connected, nontrivial and $k$-wound. Then $G$ is not special.
\end{lem}
\begin{proof} 
We proceed by induction on $\dim G$. We first suppose that $G$ is commutative, annihilated by $p$. Then $G$ is not special by Lemma~\ref{lem:01}.

Now we suppose that $G$ is either not commutative or not annihilated by $p$. Let $\kappa(G)$ be the cckp-kernel of $G$. Then by \cite[Appendix B, Proposition B.3.2]{CGP}, $G/\kappa(G)$ is $k$-wound and
\[
0< \dim G/\kappa(G) < \dim G.
\]
(Note that  the quotient $G/\kappa(G)$ is a smooth unipotent algebraic group since $\kappa(G)$ is central hence normal in $G$.)
Hence by the induction hypothesis, $G/\kappa(G)$ is not special.  Thus there exists a field extension $L$ over $k$ such that $H^1(L,G/\kappa(G))$ is nontrivial. On the other hand, by Lemma~\ref{lem:surj} the map 
\[
H^1(L,G)\to H^1(L,G/\kappa(G)) 
\]
is surjective.  This implies that $H^1(L,G)$ is nontrivial. Therefore  $G$ is not special.
\end{proof}
\subsection{Smooth non-connected unipotent groups} In this subsection we provide an alternative proof for the fact that smooth special unipotent group over a field $k$ is connected. Since every smooth unipotent group over a field of  characteristic zero is connected,  we may and shall assume that $k$ is of characteristic $p>0$. Some material in this subsection will be taken from \cite[Section 3]{T}.


Recall that the {\it Frattini subgroup} $\Phi(G)$ of an abstract finite group $G$ is the intersection of the maximal subgroups of $G$. It is a characteristic subgroup, i.e., it is invariant under every automorphism of $G$ and if $G\not= 1$ then $\Phi(G)\not=G$. If $G$ is a $p$-group then $G/\Phi(G)$ is an elementary $p$-group.

To give a finite \'etale $k$-group scheme $G$ is the same as to give a finite abstract group $\G$ with  a continuous action of $\Gal(k_s/k)$ (for example see \cite[6.4, Theorem]{Wa}. 
Note also that via this correspondence the order of $G$ is equal to the order of $\G$ as an abstract group.
Since the Frattini subgroup $\H=\Phi(\G)$ of $\G$ is invariant under the action of $\Gal(k_s/k)$, $\H$ with this Galois action defines a finite $k$-subgroup $H$ of $G$, it is also called the Frattini subgroup of $G$. If $G$ is a finite \'etale group scheme of $p$-power order, then $G/H$ is commutative and annihilated by $p$.

\begin{lem} Let $G$ be a nontrivial finite \'etale  $k$-group scheme of $p$-power order. Then $G$  is not special.
\end{lem}
\begin{proof} 
 Let $H$ be the Frattini subgroup of $G$. Then $G/H$ is a nontrivial finite \'etale  $k$-group scheme, which is commutative and annihilated by $p$. Hence $G/H$ is not special by Lemma~\ref{lem:01}.  Thus there exists a field extension $L$ over $k$ such that $H^1(L,G/H)$ is nontrivial. Then Lemma~\ref{lem:surj} implies that  $H^1(L,G)$ is nontrivial, and hence $G$ is not special.
\end{proof}
\begin{lem}
\label{lem:non-connected} Let $G$ be a smooth non-connected unipotent $k$-group. Then $G$ is not special.
\end{lem}
\begin{proof} Let $G^\circ$ be the connected component of $G$. Then $G/G^\circ$ is a nontrivial finite \'etale $k$-group scheme (for example, see \cite[6.7, Theorem]{Wa}. Because $G/G^{\circ}$ is  unipotent, it is also of $p$-power order. Hence $G/G^\circ$ is not special. Thus there exists a field extension $L$ over $k$ such that $H^1(L,G/G^\circ)$ is nontrivial. Then Lemma~\ref{lem:surj} implies that  $H^1(L,G)$ is nontrivial, and hence $G$ is not special.
\end{proof}
\subsection{Proof of Theorem~\ref{thm:main}}
\begin{proof}
We may assume that $k$ is imperfect. 
Suppose that $G$ is special. By \cite[\S 4, Theorem 1]{Se1} or Lemma~\ref{lem:non-connected}, $G$ is connected. 
Let $G_s$ be the $k$-split part of $G$. Then the quotient $G/G_s$ is $k$-wound. On the other hand, Lemma~\ref{lem:surj} implies that $G/G_s$ is  special. Hence $G/G_s$ is trivial by Lemma~\ref{lem:02}. Therefore $G=G_s$ is $k$-split.
\qedhere
\end{proof}
\begin{rmk} One may define that for any group scheme $G$ of finite type over $k$, $G$ is {\it special} if and only if $H^1_{fppf}(K,G)=1$ for every  field extension $K$ over $k$. Then 
\cite[Theorem 1.2]{TV} implies that if $G$ is special, then $G$ is smooth. 
\end{rmk}


\begin{thebibliography}{9999999}
\bibitem[A]{A} R. Achet, {\it 
Picard group of the forms of the affine line and of the additive group},  to appear in  Journal of Pure and Applied Algebra.
\bibitem[Bo]{Bo} A. Borel, {\it Linear Algebraic Groups (2nd ed.)}, Graduate texts in mathematics {\bf 126}, New York: Springer-Verlag 1991.
\bibitem[CGP]{CGP} B. Conrad, O. Gabber and G. Prasad, {\it Pseudo-reductive groups}, Series: New Mathematical Monographs (No. 17), Cambrigde University Press, Cambridge, 2010.
\bibitem[GM]{GM} P. Gille and L. Moret-Bailly, {\it Actions alg\'ebriques de groupes arithm\'etiques},  Torsors, \'etale homotopy and applications to rational points, 231-249, London Math. Soc. Lecture Note Ser., 405, Cambridge Univ. Press, Cambridge, 2013. .
\bibitem[Gro]{Gro} A. Grothendieck, {\it Torsion homologique et sections rationnelles}, Anneaux de Chow et Applications, S\'eminaire Claude Chevalley, 1958, expos\'e n. 5.

\bibitem[Hu]{Hu} M. Huruguen, {\it Special reductive groups over an arbitrary field},  Transform. Groups 21 (2016), 1079-1104.
\bibitem[Oe]{Oe} J. Oesterl\'e, {\it Nombre  de Tamagawa et groupes unipotents en  caract\'eristique $p$},  Invent. Math. {\bf 78} (1984), 13-88.
\bibitem[Se]{Se1} J.-P. Serre, {\it Espaces fibr\'es alg\'ebriques}, Anneaux de Chow et Applications, S\'eminaire Claude Chevalley, 1958, expos\'e n. 1.
\bibitem[T]{T} {\fontencoding{T5} \selectfont Nguy\~ \ecircumflex n}  D. T\^an, {\it On the essential dimension of unipotent algebraic groups}, J. Pure Appl. Algebra 217 (2013), no. 3, 432-448.
\bibitem[TV]{TV} D. Tossici and A. Vistoli, {\it On the essential dimension of infinitesimal group schemes},  Amer. J. Math. 135 (2013), no. 1, 103-114.
\bibitem[Wa]{Wa} W. Waterhouse, {\it Introduction to affine group schemes }, Graduate Texts in Mathematics 66 (1979), Springer-Verlag, New York.
\end{thebibliography}
\end{document}